\documentclass[12pt]{amsart}

\usepackage{amsthm, amsmath, amssymb}
\usepackage{enumerate}
\usepackage{thmtools}

\usepackage{color}
\definecolor{darkgreen}{rgb}{0,0.35,0}
\definecolor{darkblue}{rgb}{0,0,0.6}
\usepackage[colorlinks,citecolor=darkgreen,linkcolor=darkblue]{hyperref}

\declaretheorem[numberwithin=section]{theorem}
\declaretheorem[numberlike=theorem]{lemma}
\declaretheorem[numberlike=theorem]{corollary}
\declaretheorem[numberlike=theorem]{proposition}

\newtheoremstyle{parentheses}{\medskipamount}{\medskipamount}{\itshape}{}{}{}{ }{\thmnumber{(#2)}}
\theoremstyle{parentheses}
\newtheorem{claim}{}
\makeatletter\@addtoreset{claim}{theorem}\makeatother

\renewcommand{\epsilon}{\varepsilon}
\renewcommand{\phi}{\varphi}
\newcommand{\spn}[1]{\langle #1\rangle}

\allowdisplaybreaks
\title[Extremal functions for matroids]{The extremal functions of classes of matroids of bounded branch-width}
\author{Rohan Kapadia}
\address{Department of Computer Science and Software Engineering, Concordia University, Montreal, Quebec, Canada}
\email{rohan.f.kapadia@gmail.com}

\subjclass[2010]{05B35}
\date{March 26, 2015; revised January 24, 2016}

\begin{document}

\begin{abstract}
For a set of matroids $\mathcal{M}$, let $ex_\mathcal{M}(n)$ be the maximum size of a simple rank-$n$ matroid in $\mathcal{M}$.
We prove that, for any finite field $\mathbb{F}$, if $\mathcal{M}$ is a minor-closed class of $\mathbb{F}$-representable matroids of bounded branch-width, then $\lim_{n \rightarrow \infty} ex_\mathcal{M}(n) / n$ exists and is a rational number, $\Delta$.
We also show that $ex_\mathcal{M}(n) - \Delta n$ is periodic when $n$ is sufficiently large and that $ex_\mathcal{M}$ is achieved by a subclass of $\mathcal{M}$ of bounded path-width.
\end{abstract}

\maketitle

\section{Introduction}

A classic theorem of extremal graph theory is Tur\'an's theorem, which tells us the maximum number of edges in a simple $n$-vertex graph with no $K_r$ subgraph, and determines the graphs achieving the maximum.
Much recent work has gone into the related extremal problem for graph minors: given a proper minor-closed class of graphs $\mathcal{G}$, what is the maximum number $ex_\mathcal{G}(n)$ of edges in a simple $n$-vertex graph in $\mathcal{G}$?
It was first proved by Mader \cite{Mader} that this number is bounded by a linear function of $n$.
The exact extremal function is known for several particular classes of graphs; see for example \cite{SongThomas}.
The best-known case is that of the planar graphs $\mathcal{P}$, where $ex_\mathcal{P}(n) = 3n-6$ for $n \geq 3$.
A more interesting example is the class $\mathcal{G}$ of graphs with no $K_{3,3}$-minor, for which we notice a certain periodic behaviour for $n \geq 2$:
\[ ex_\mathcal{G}(n) = \begin{cases} 3n - 5, & \text{ if } n \equiv 2 \pmod 3 \\
                                     3n - 6, & \text{ otherwise.}
                       \end{cases} \]
This example illustrates the general principle governing extremal functions of minor-closed classes.
In recent work with Sergey Norin~\cite{KapadiaNorin} we show, for any proper minor-closed class of graphs $\mathcal{G}$, that $\lim_{n \rightarrow \infty} ex_\mathcal{G}(n)/n$ exists and is a rational number, $\Delta$, and that $ex_\mathcal{G}(n) - \Delta n$ is periodic when $n$ is large enough, and we characterize certain extremal graphs.
In this paper, we take the first step towards extending these facts from graphs to matroids.
In fact, the techniques we use are actually matroidal versions of the methods used for minor-closed classes of graphs of bounded tree-width in \cite{KapadiaNorin}.

For a matroid $M$, we write $\epsilon(M)$ for the number of points (rank-one flats) in $M$, or equivalently, the size of the simplification of $M$.
We define the \emph{extremal function} $ex_\mathcal{M}$ for a set of matroids $\mathcal{M}$ by setting
\[ ex_\mathcal{M}(n) = \max\{\epsilon(M) : M \in \mathcal{M}, r(M) = n\}, \]
where the function takes the value $\infty$ if the maximum does not exist.
A class of matroids $\mathcal{M}$ is called \emph{linearly dense} if there is a number $c$ such that $ex_\mathcal{M}(n) \leq c n$ for all $n \geq 0$.
It was proven in Geelen and Whittle \cite{GeelenWhittle} that a minor-closed class is linearly dense if and only if it does not contain all simple rank-two matroids and does not contain all graphic matroids.
We will focus on a particular type of linearly dense class.
Given a finite field $\mathbb{F}$, we look at minor-closed classes of $\mathbb{F}$-representable matroids of bounded branch-width (branch-width will be defined later).
These classes are linearly dense because the set of graphic matroids has unbounded branch-width. In fact, Geelen, Gerards and Whittle \cite{GeelenGerardsWhittle2007} have shown that a minor-closed class of $\mathbb{F}$-representable matroids has bounded branch-width if and only if it does not contain all the planar graphic matroids.

We prove the following theorem, which, along with \autoref{thm:path} that appears in the last section, confirms special cases of Conjectures~7.6, 7.7, and 7.8 of \cite{GeelenGerardsWhittle2015}.

\begin{theorem} \label{thm:maintheorem}
 For each finite field $\mathbb{F}$ and each minor-closed class $\mathcal{M}$ of $\mathbb{F}$-representable matroids of bounded branch-width, there are integers $p$ and $m$ and rational numbers $\Delta$ and $a_0, \ldots, a_{p-1}$ such that $ex_\mathcal{M}(n) = \Delta n + a_i$ whenever $n \equiv i \pmod p$ and $n > m$.
\end{theorem}

We will prove this theorem by finding a structural characterization of some matroids of extremal size. We show that the extremal size is always attained by a subclass of matroids with a certain path-like decomposition. In fact, this subclass has bounded path-width; we will not use path-width in this paper, but see \cite{GeelenGerardsWhittle2006} for a definition.

The number $\Delta$ given by \autoref{thm:maintheorem} is known as the limiting density of the class $\mathcal{M}$. Eppstein \cite{Eppstein} began a study of the possible values of limiting densities of minor-closed classes of graphs and posed several questions about them.

In matroid theory literature, the extremal function is often called the \emph{growth-rate function} of the class $\mathcal{M}$ and denoted by $h_\mathcal{M}$ or $h(\mathcal{M}, \cdot)$. We are using the graph-theoretic terminology here because of the close connection between the concept for linearly dense matroids and graphs.

In the next three sections, we present the notions of matroid density, of configurations, and of branch decompositions.
In \autoref{sec:rational}, we prove that, for any finite field $\mathbb{F}$, the limiting density of any minor-closed class of $\mathbb{F}$-representable matroids of bounded branch-width is rational, and in the last section we complete the proof of \autoref{thm:maintheorem}.

\section{Density}

The \emph{density} of a matroid $M$ is $d(M) = \epsilon(M) / r(M)$.
So a minor-closed class of matroids $\mathcal{M}$ is linearly dense if there is a number $c$ such that $d(M) \leq c$ for every matroid $M$ in $\mathcal{M}$.
The \emph{limiting density} of a linearly dense class $\mathcal{M}$, denoted $d(\mathcal{M})$, is the minimum real number $d$ such that any rank-$n$ matroid in $\mathcal{M}$ has density at most $(d + o(1))n$ (this is analogous to the limiting density of a class of graphs, as defined in Eppstein \cite{Eppstein}).
That is,
\[ d(\mathcal{M}) = \limsup_{n \rightarrow \infty} \left(\max\{d(M) : M \in \mathcal{M}, r(M) = n\}\right). \]

Let $k$ be a positive integer and $\delta$ a positive real number. A matroid $M$ is called \emph{$(\delta, k)$-pruned} if, for every minor $N$ of $M$ with rank at least $r(M) - k$, we have
\[ \epsilon(M) - \epsilon(N) \geq (d(M) - \delta) (r(M) - r(N)) .\]

We say that a sequence of matroids $\{M_i : i \geq 1\}$ is \emph{pruned} if, for every positive real number $\delta$ and every positive integer $k$, there exists an integer $m$ so that, for all $n \geq m$, $M_n$ is $(\delta, k)$-pruned.

\begin{lemma} \label{lem:prunedsequenceexists}
 If $\mathcal{M}$ is a linearly dense minor-closed class of matroids with limiting density $\Delta$ and $\Delta > 0$, then there is a pruned sequence $\{M_i : i \geq 1\}$ of matroids in $\mathcal{M}$ such that $d(M_i) \rightarrow \Delta$ and $r(M_i) \rightarrow \infty$.
\end{lemma}

\begin{proof}
 Since $\Delta > 0$, there is a sequence of matroids $\{M'_i : i \geq 1\}$ in $\mathcal{M}$ such that $d(M'_i) \rightarrow \Delta$ and $r(M'_i) \rightarrow \infty$.
 
 Let $\{\delta_i : i \geq 1\}$ be a decreasing sequence of positive real numbers that converges to zero.
 Let $\{k_i : i \geq 1\}$ be a strictly increasing sequence of positive integers.
 For each $i$, there is a positive integer $m_i$ such that any matroid in $\mathcal{M}$ with rank at least $m_i$ has density at most $\Delta + \delta_i / 4$. Set $n_i = m_i + k_i$ for each $i \geq 1$.
 
 Consider some pair $(\delta_i, k_i)$. Let $c_i$ be the maximum number of points in any matroid in $\mathcal{M}$ with rank at most $n_i$.
 We can pick an integer $h$ such that $d(M'_h) > \Delta - \delta_i / 4$ and $r(M'_h) > \max\{n_i, 2c_i / \delta_i\}$.
 We shall show that $M'_h$ has an $(\delta_i, k_i)$-pruned minor with rank at least $n_i$.
 We pick a maximal sequence $(N_0, \ldots, N_t)$ of minors of $M'_h$ where
 \begin{enumerate}[(a)]
  \item $N_0 = M'_h$,
  \item $N_j$ is a minor of $N_{j-1}$ with rank at least $r(N_{j-1}) - k_i$, for each $j = 1, \ldots, t$,
  \item $\epsilon(N_{j-1}) - \epsilon(N_j) < (d(N_{j-1}) - \delta_i)(r(N_{j-1}) - r(N_j))$, for each $j = 1, \ldots, t$, and \label{pruning}
  \item $r(N_{t-1}) \geq n_i$.
 \end{enumerate}
 We may assume that $M'_h$ is not $(\delta_i, k_i)$-pruned, so $t \geq 1$.
 Note that (\ref{pruning}) implies that $d(N_t) > \cdots > d(N_0)$. Thus we have 
 \[ \epsilon(N_0) - \epsilon(N_t) \leq \sum_{j = 1}^t (d(N_t) - \delta_i) (r(N_{j-1}) - r(N_j)) \]
 which means
 \[ \epsilon(M'_h) - \epsilon(N_t) \leq (d(N_t) - \delta_i)(r(M'_h) - r(N_t)) .\]
 Suppose that $r(N_t) < n_i$. Then,
 \[ \epsilon(M'_h) - c_i \leq (d(N_t) - \delta_i)r(M'_h) .\]
 On the other hand, since $r(N_t) \geq r(N_{t-1}) - k_i \geq n_i - k_i = m_i$, we know that $d(N_t) \leq \Delta + \delta_i / 4$.
 But $\Delta < d(M'_h) + \delta_i / 4$, so that
 \[ \epsilon(M'_h) - c_i < (d(M'_h) - \delta_i / 2)r(M'_h) .\]
 Equivalently, $\frac{\delta_i}{2} r(M'_h) < c_i$, which contradicts our choice of $h$.
 This proves that $r(N_t) \geq n_i$. Then the maximality of the sequence $N_0, \ldots, N_t$ implies that $N_t$ is $(\delta_i, k_i)$-pruned.
 Let $M_i = N_t$; as we observed above, $d(M_i) \geq d(M'_h) > \Delta - \delta_i / 4$.
 
 Now $\{M_i : i \geq 1\}$ is a sequence of matroids such that $r(M_i) \rightarrow \infty$, $d(M'_i) \rightarrow \Delta$, and $M_i$ is $(\delta_i, k_i)$-pruned for each $i \geq 1$.
 
 The lemma now follows from the fact that for any $\delta, \delta_i > 0$ and positive integers $k, k_i$, if $\delta_i \leq \delta$ and $k_i \geq k$ then any $(\delta_i, k_i)$-pruned matroid is $(\delta, k)$-pruned.
\end{proof}

\section{Configurations}

We present some definitions that partly come from \cite{GeelenGerardsWhittle} but with some modifications. 
Let $\mathbb{K}$ be a field.
A \emph{configuration} is a finite multiset of elements of some $\mathbb{K}$-vector space.
A \emph{subconfiguration} of a configuration $A$ is a configuration that is contained in $A$.
The linear span of a configuration $A$ is denoted $\spn{A}$.

A configuration $A$ is called a \emph{minor} of a configuration $A'$ if there is a linear transformation $\mathcal{L}$ from $\spn{A'}$ to $\spn{A}$ such that $\spn{A} = \mathcal{L}(\spn{A'})$, $\ker(\mathcal{L})$ is the linear span of some subset of $A'$, and $A \subseteq \mathcal{L}(A')$. When this holds, we write $A \overset{\mathcal{L}}{\leftarrow} A'$.

The matroid $M(A)$ \emph{represented} by a configuration $A$ is the matroid with ground set $A$ in which independence is linear independence over $\mathbb{K}$. The following is Theorem~5.4 of \cite{GeelenGerardsWhittle}.

\begin{proposition}
 If $A \overset{\mathcal{L}}{\leftarrow} A'$, then $M(A)$ is obtained from $M(A')$ by contracting a subset $X$ of $\ker(\mathcal{L}) \cap A'$ that spans $\ker(\mathcal{L})$, adding back a loop for each member of $X$, and finally taking the restriction to those elements of $A'$ mapped by $\mathcal{L}$ to $A$. Conversely, for each minor $M$ of $M(A')$, there exists a linear transformation $\mathcal{L}$ and a configuration $A$ such that $M$ is equal to $M(A)$ and $A \overset{\mathcal{L}}{\leftarrow} A'$.
\end{proposition}

This means that the minor relation on matroids over $\mathbb{K}$ is the same as that on configurations over $\mathbb{K}$, if we ignore the presence of loops and zero vectors. We can therefore work with configurations in place of matroids, since loops are irrelevant to questions of density.

We can extend all the notions of density from matroids to configurations. For a configuration $A$, we define $\epsilon(A) = \epsilon(M(A))$ and $d(A) = d(M(A))$, so $d(A) = \epsilon(A) / \dim(\spn{A})$. For a set $\mathcal{F}$ of configurations, the limiting density of $\mathcal{F}$ is that of the set of matroids $\{M(A) : A \in \mathcal{F}\}$.
We also define the extremal function $ex_\mathcal{F}$ to be that of this corresponding set of matroids. So $ex_\mathcal{F}(n) = \max\{\epsilon(A) : A \in \mathcal{F}, \dim(\spn{A}) = n\}$.

\subsection*{Rooted configurations and patches}

We call a triple $(A, L, R)$ of configurations a \emph{rooted configuration} if there is a configuration $A^*$ that can be partitioned into subconfigurations $A$, $L$, and $R$ such that the sets $L$ and $R$ are both linearly independent in $\spn{A^*}$, $R \subseteq \spn{A \cup L}$, and $L \subseteq \spn{A \cup R}$. We treat $L$ and $R$ as sequences, so their elements have an ordering $L = \{l_1, \ldots, l_{|L|}\}$ and $R = \{r_1, \ldots, r_{|R|}\}$. We call their elements the \emph{left terminals} and the \emph{right terminals} of the rooted configuration, respectively.
For a rooted configuration $H = (A, L, R)$, we write $\widetilde{H}$ to denote the configuration $A$.
Also, to avoid complicated notation, we write $\spn{H}$ for $\spn{A \cup L \cup R}$.
We call a rooted configuration $H$ \emph{spanning} if $\dim(\spn{\widetilde{H}}) = \dim(\spn{H})$.
We call $H = (A, L, R)$ \emph{non-trivial} if $\dim(\spn{H}) > \dim(\spn{L}) = |L|$.

An \emph{isomorphism} between two rooted configurations $H_1 = (A_1, L_1, R_1)$ and $H_2 = (A_2, L_2, R_2)$ is an isomorphism between $\spn{H_1}$ and $\spn{H_2}$ that maps $A_1$ onto $A_2$ and maps the elements of $L_1$ and $R_1$ onto those of $L_2$ and $R_2$, in order.

We call a rooted configuration $H = (A, L, R)$ a \emph{minor} of another one $H' = (A', L', R')$ if there is a linear transformation $\mathcal{L}$ from $\spn{H'}$ to $\spn{H}$ such that $A \overset{\mathcal{L}}{\leftarrow} A'$ (so $\ker(\mathcal{L})$ is the span of some subset of $A'$) and $\mathcal{L}$ maps the elements of $L'$ and $R'$ respectively onto the elements of $L$ and $R$, in order. We write $H \overset{\mathcal{L}}{\leftarrow} H'$.

Let $q$ be a non-negative integer. We define a \emph{$(\leq q)$-rooted configuration} to be a rooted configuration $(A, L, R)$ where $|L| \leq q$ and $|R| \leq q$ and we call it a \emph{$q$-patch} if $|L| = q$ and $|R| = q$.
A $q$-patch $(A, L, R)$ is called \emph{linked} if it has a minor $(A', L', R')$ such that $L'$ and $R'$ are equal as ordered sequences.

\subsection*{Products}

Let $(A, L, R)$ be a $(\leq q)$-rooted configuration. Let $(A_1, A_2)$ be a partition of $A$ into two sets such that $\dim(\spn{A_1 \cup L} \cap \spn{A_2 \cup R}) \leq q$.
Let $X$ be a basis of $\spn{A_1 \cup L} \cap \spn{A_2 \cup R}$. Then we can define the $(\leq q)$-rooted configurations $(A_1, L, X)$ and $(A_2, X, R)$.
We say that $(A, L, R)$ is the \emph{product} of $(A_1, L, X)$ and $(A_2, X, R)$ and we write $(A, L, R) = (A_1, L, X) \times (A_2, X, R)$.

A product is a way to decompose a rooted configuration into two pieces, but we also need a way to compose two rooted configurations into a product when they aren't necessarily contained in the same underlying vector space. 
However, this cannot always be defined uniquely.
Let $H_1 = (A_1, L_1, R_1)$ and $H_2 = (A_2, L_2, R_2)$ be two rooted configurations.
We define $\mathcal{P}(H_1, H_2)$ to be the set of all rooted configurations $H'_1 \times H'_2$ where $H'_1$ is isomorphic to $H_1$ and $H'_2$ is isomorphic to $H_2$.
This set is only non-empty when $|R_1| = |L_2|$ and there is an isomorphism between the spaces $\spn{R_1}$ and $\spn{L_2}$ that maps the elements of $R_1$, in order, to those to $L_2$.

More generally, we write $\mathcal{P}(H_1, \ldots, H_k)$ for the set of all rooted configurations $H'_1 \times \cdots \times H'_k$ where $H'_i$ is isomorphic to $H_i$ for each $i = 1, \ldots, k$.
We call all rooted configurations in this set \emph{products} of $H_1, \ldots, H_k$.
When $H$ is a $q$-patch we write $\mathcal{P}(H^k)$ for $\mathcal{P}(H, \ldots, H)$, the set of products of $k$ copies of $H$.
Products and linked $q$-patches are useful because of the following.

\begin{proposition} \label{prop:contractinglinkedpatch}
 If $H_1, H_2$, and $H_3$ are $q$-patches and $H_2$ is linked, then every element of $\mathcal{P}(H_1, H_2, H_3)$ has a minor in $\mathcal{P}(H_1, H_3)$.
\end{proposition}

\begin{proof}
Let $H$ be an element of $\mathcal{P}(H_1, H_2, H_3)$.
Write $H_2 = (\widetilde{H_2}, L_2, R_2)$.
There is a linear transformation on $\spn{H_2}$ whose kernel is the span of a subset of $\widetilde{H_2}$ that maps the elements of $R_2$ in order onto those of $L_2$.
We can apply the same linear transformation to the copy of $\spn{H_2}$ in $\spn{H}$, and then extend this linear transformation to a linear transformation $\mathcal{L}$ on $\spn{H}$ whose kernel is the span of a subset of the copy of $\widetilde{H_2}$.
The minor $H'$ of $H$ such that $H' \overset{\mathcal{L}}{\leftarrow} H$ is in $\mathcal{P}(H_1, H_3)$.
\end{proof}

A second useful property of products is that whenever $G$ and $H$ are rooted configurations and $G$ is spanning, any element of $\mathcal{P}(G, H)$ is also a spanning rooted configuration.

\section{Branch decompositions}

Recall that the connectivity function $\lambda_M$ of a matroid $M$ is defined for sets $X \subseteq E(M)$ by $\lambda_M(X) = r_M(X) + r_M(E(M) \setminus X) - r(M)$.
For a configuration $A$ and a subset $X$ of $A$, note that $\lambda_{M(A)}(X) = \dim(\spn{X} \cap \spn{A \setminus X})$.

A \emph{branch decomposition} of a matroid $M$ is a tree $T$ where every vertex has degree one or three and $E(M)$ is a subset of the leaves of $T$.
The set \emph{displayed} by a subtree of $T$ is the set of elements of $E(M)$ in that subtree. A subset $X$ of $E(M)$ is displayed by an edge $e$ of $T$ if it is displayed by one of the components of $T - e$. The \emph{width} of $e$, denoted $\lambda(e)$, is the value of $\lambda_M(X)$ where $X$ is any of the sets displayed by $e$. The \emph{width} of a branch decomposition is the maximum of the widths of its edges and the \emph{branch-width} of a matroid is the smallest of the widths of all its branch decompositions.

We define a branch decomposition of a configuration $A$ to be a branch decomposition of the matroid $M(A)$ and the branch-width of $A$ to be that of $M(A)$.
For a rooted configuration $H$, we define the branch-width of $H$ to be that of $\widetilde{H}$.
It was proved by Geelen, Gerards and Whittle \cite{GeelenGerardsWhittle} that configurations over a finite field with bounded branch-width are well-quasi-ordered by the minor relation.

\begin{theorem}[{\cite[Theorem 5.8]{GeelenGerardsWhittle}}] \label{thm:wqo}
For any finite field $\mathbb{F}$ and natural number $n$, the set of configurations over $\mathbb{F}$ with branch-width at most $n$ is well-quasi-ordered by the minor relation.
\end{theorem}

A $q$-patch is essentially a configuration with $2q$ distinguished elements.
So we can extend \autoref{thm:wqo} from configurations to $q$-patches by `marking' a set of $2q$ distinguished elements of a configuration. We can do this by going to a larger finite field and gluing non-$\mathbb{F}$-representable matroids onto these elements.

\begin{theorem} \label{thm:rootedwqo}
For any finite field $\mathbb{F}$ and natural numbers $n$ and $q$, the set of $q$-patches over $\mathbb{F}$ with branch-width at most $n$ is well-quasi-ordered by the minor relation.
\end{theorem}

\begin{proof}
 Let $C_1, C_2, \ldots$ be an infinite sequence of $q$-patches over $\mathbb{F}$. We need to show that there are indices $i, j$ with $i < j$ such that $C_i$ is a minor of $C_j$.
 Let $\mathbb{F}'$ be a finite extension field of $\mathbb{F}$ such that $|\mathbb{F}'| \geq |\mathbb{F}| + 2q$.
 We can view the $q$-patches $C_1, C_2, \ldots$ as $q$-patches over $\mathbb{F}'$ (by applying, component-wise to each vector in the configuration, an embedding of $\mathbb{F}$ onto a subfield of $\mathbb{F}'$).
 For each $C_i = (\widetilde{C_i}, L_i, R_i)$, we let $A_i$ be the configuration $\widetilde{C_i} \cup L_i \cup R_i$. 
 Denote the $j$th element of $L_i$ by $l_j$ and the $j$th element of $R_i$ by $r_j$, for each $j = 1, \ldots, q$.
 We define $M_i$ to be the matroid obtained from $M(A_i)$ by taking repeated $2$-sums as follows. For each $j = 1, \ldots, q$ we do a $2$-sum  with a copy of $U_{2, |\mathbb{F}| + 1 + j}$ with basepoint $l_j$. For each $j = 1, \ldots, q$ again, we do a $2$-sum with a copy of $U_{2, |\mathbb{F}| + q + 1 + j}$ with basepoint $r_j$.
 We do all the $2$-sums without deleting the basepoints.
 Note that none of these lines are representable over $\mathbb{F}$.
 
 Since the $\mathbb{F}'$-representable matroids of branch-width at most $n$ are well-quasi ordered by the minor relation, there are indices $i, j$ with $i < j$ such that $M_i$ is (isomorphic to) a minor $M_j$.
 No elements of the lines we added by $2$-summing can be deleted or contracted from $M_j$ to get $M_i$. So there is a set $X$ in $\widetilde{C_j}$ such that $M_i$ is isomorphic to a restriction of $M_j / X$, by an isomorphism that maps the elements of $L_i$ and $R_i$ to those of $L_j$ and $R_j$, in order.
 The $q$-patch $C_i$ is a minor of the $q$-patch $C_j$.
\end{proof}

\subsection*{Linked branch decompositions}

For two disjoint sets $A, B$ in a matroid $M$, we write $\kappa_M(A, B)$ for the minimum of $\lambda_M(X)$ over all sets $X$ containing $A$ and disjoint from $B$. Clearly, $\kappa_M(A, B) = \kappa_M(B, A)$.

Let $f$ and $g$ be two edges in a branch decomposition $T$ of $M$, let $F$ be the set displayed by the component of $T - f$ not containing $g$, and let $G$ be the set displayed by the component of $T - g$ not containing $f$. Let $P$ be the shortest path of $T$ containing $f$ and $g$. 
The edges $f$ and $g$ are called \emph{linked} if $\kappa_M(F, G)$ is equal to the minimum width of the edges of $P$.
The branch decomposition $T$ is called \emph{linked} if all edge pairs are linked.
It was proved in Geelen, Gerards and Whittle \cite{GeelenGerardsWhittle} that we can always find linked branch decompositions:

\begin{theorem}[{\cite[Theorem 2.1]{GeelenGerardsWhittle}}] \label{thm:everymatroidhaslinkedbranchdecomposition}
 Any matroid of branch-width $n$ has a linked branch decomposition of width $n$.
\end{theorem}

We can always choose such a linked branch-decomposition so that every leaf of it is actually an element of the matroid.
As we shall see, linked branch decompositions are useful because of Tutte's Linking Theorem (see \cite[Theorem 5.1]{GeelenGerardsWhittle} for a proof):

\begin{theorem}[Tutte's Linking Theorem]
 If $X$ and $Y$ are disjoint subsets in a matroid $M$, then $\kappa_M(X, Y) \geq n$ if and only if there exists a minor $M'$ of $M$ with ground set $X \cup Y$ such that $\lambda_{M'}(X) \geq n$.
\end{theorem}

\subsection*{Rooted branch decompositions}

A \emph{rooted tree} is a tree whose edges are oriented such that it has precisely one vertex, called the \emph{root}, with indegree zero. The \emph{parent} of a vertex in a rooted tree is its neighbour on the path joining it to the root.
We define the \emph{depth} of a rooted tree to be the maximum distance between a leaf and the root.
A \emph{rooted branch decomposition} of a configuration $A$ is a branch decomposition that is a rooted tree.
Every configuration $A$ of branch-width $n$ has a rooted, linked branch decomposition of width $n$.

\subsection*{Decomposing into a product}

In this subsection, we show that any large enough configuration of bounded branch-width can be written as a product of rooted configurations in a certain way.
When $A'$ is a subconfiguration of a configuration $A$, the \emph{boundary of $A'$ in $A$} is the space $\spn{A'} \cap \spn{A \setminus A'}$.
So the dimension of the boundary of $A'$ is equal to $\lambda_{M(A)}(A')$.

\begin{lemma} \label{lem:decomposeintoproduct0}
 For any positive integers $w$ and $p$ and any configuration $A$ with branch-width at most $w$ such that $|A| > 2^p$, there is a product of $p$ $(\leq w)$-rooted configurations
 \[ (A, L_1, R_p) = (A_1, L_1, R_1) \times \cdots \times (A_p, L_p, R_p) \]
 such that $(A_1, L_1, R_1)$ is spanning, and for all $i = 1, \ldots, p - 1$, $1 \leq |A_i| \leq 2^{i-1}$ and $R_i$ spans the boundary of $A_1 \cup \cdots \cup A_i$ in $A$.
 Moreover, $\kappa_{M(A)}(A_1 \cup \cdots \cup A_i, A_j \cup \cdots \cup A_p) \geq \min\{|R_i|, |R_{i+1}|, \ldots, |R_{j-1}|\}$ for any $i < j$.
\end{lemma}

\begin{proof}
 Let $A$ be a configuration over some field with branch-width at most $w$ and $|A| \geq 2^p$.
 By \autoref{thm:everymatroidhaslinkedbranchdecomposition}, it has a linked, rooted branch decomposition $T$ of width at most $w$. We can choose it so that every leaf of $T$ is an element of $A$.

 If $T$ has depth less than $p$, then it has fewer than $2^p$ leaves, so $|A| < 2^p$, a contradiction. So $T$ has depth at least $p$.
 We pick a vertex $v_1$ at maximum distance from the root, and consider the set of vertices $\{v_1, \ldots, v_p\}$ where $v_i$ is the parent of $v_{i-1}$ in $T$, for each $i = 2, \ldots, p$. Let $P$ denote the $v_1, v_p$-path of $T$ and write $e_i$ for the edge of $P$ joining $v_i$ to $v_{i+1}$, $i = 1, \ldots, p - 1$.
 For each $i = 1, \ldots, p$, we define $S_i$ to be the maximal subtree of $T$ containing $v_i$ but no other vertex of $P$ and we set $A_i$ to be the set displayed by $S_i$. Then the sets $A_1, \ldots, A_p$ partition $A$ and, for each $i$, the dimension of the boundary of $A_1 \cup \cdots \cup A_i$ is the width of the edge $e_i$ in the branch decomposition, which is at most $w$.

 We pick a basis $R_1$ of the boundary of $A_1$ in $A$ and set $L_1 = R_1$; then $(A_1, L_1, R_1)$ is a $(\leq w)$-rooted configuration and it is spanning.
 Since $S_1$ is a one-vertex tree (the leaf $v_1)$, we have $|A_1| = 1$. 
 For each $i = 2, \ldots, p - 1$, we inductively set $L_i = R_{i-1}$ and let $R_i$ be a basis of the boundary of $A_1 \cup \cdots \cup A_i$ in $A$; then $(A_i, L_i, R_i)$ is a $(\leq w)$-rooted configuration. 
 Finally, we set $R_p$ and $L_p$ equal to $R_{p-1}$ so $(A_p, L_p, R_p)$ is a $(\leq w)$-rooted configuration.
 We have $(A, L_1, R_p) = (A_1, L_1, R_1) \times \cdots \times (A_p, L_p, R_p)$.
 
 The fact that $v_1$ is a leaf of $T$ at the maximum distance from the root means that $S_i$ is a tree of depth at most $i-1$ and so $|A_i| \leq 2^{i-1}$. Since $T$ has no vertex of degree two, every tree $S_i$ has a leaf so $|A_i| \geq 1$. 
 
 For any $i < j$, the set $A_1 \cup \cdots \cup A_i$ is displayed by the edge $e_i$ and the set $A_j \cup \cdots \cup A_p$ is displayed by the edge $e_{j-1}$.
 Thus, since $T$ is a linked branch decomposition, the value of $\kappa_{M(A)}(A_1 \cup \cdots \cup A_i, A_j \cup \cdots \cup A_p)$ equals the minimum width of the edges $e_i, \ldots, e_{j-1}$, and these widths are equal to $|R_i|, \ldots, |R_{j-1}|$.
\end{proof}

We can strengthen the above lemma for finite fields to get a product of non-trivial rooted configurations.

\begin{lemma} \label{lem:decomposeintoproduct}
 For any positive integers $w$ and $p$ and any configuration $A$ with branch-width at most $w$ over a finite field $\mathbb{F}$ such that $\epsilon(A) > 2^{(|\mathbb{F}|^w+1)p}$, there is a product of $p$ non-trivial $(\leq w)$-rooted configurations
 \[ (A, L_1, R_p) = (A_1, L_1, R_1) \times \cdots \times (A_p, L_p, R_p) \]
 such that $(A_1, L_1, R_1)$ is spanning and $|A_i| \leq 2^{(|\mathbb{F}|^w+1)i}$ for all $i = 1, \ldots, p - 1$.
 Moreover, $\kappa_{M(A)}(A_1 \cup \cdots \cup A_i, A_j \cup \cdots \cup A_p) \geq \min\{|R_i|, |R_{i+1}|, \ldots, |R_{j-1}|\}$ for any $i < j$.
\end{lemma}

\begin{proof}
 We may assume that $M(A)$ is simple, that is, the multiset $A$ does not have two copies of any vector.
 \autoref{lem:decomposeintoproduct0} gives us a product of $p' = (|\mathbb{F}|^w+1)p$ possibly trivial $(\leq w)$-rooted configurations
 \[ (A, L_1, R_{p'}) = (A_1, L_1, R_1) \times \cdots \times (A_{p'}, L_{p'}, R_{p'}). \]
 Write $H_i = (A_i, L_i, R_i)$ for each $i$.
 We shall combine these into larger rooted configurations $H'_1, \ldots, H'_p$ that satisfy the lemma.
 
 Note that if $H_i \times \cdots \times H_j$ is trivial for some $i < j$, then $j \leq i + |\mathbb{F}|^w$, because each $|\spn{R_i}| \leq |\mathbb{F}|^w$.
 Hence, since $p' = (|\mathbb{F}|^w+1)p$, there are at least $p$ non-trivial terms in the sequence $H_1, \ldots, H_{p'}$.
 We set $\ell_1, \ldots, \ell_{p-1}$ such that $H_{\ell_1}, \ldots, H_{\ell_{p-1}}$ are the first $p-1$ non-trivial members of the sequence.
 We have $\ell_i \leq (\mathbb{F}|^w+1)i$ for each $i$.
 
 We define $H'_1 = H_1 \times \cdots \times H_{\ell_1}$. For each $i = 2, \ldots, p-1$ we define $H'_i = H_{\ell_{i-1}+1} \times \cdots \times H_{\ell_i}$, and we define $H'_p = H_{\ell_{p-1}+1} \times \cdots \times H_{p'}$.
 Write $H'_i = (A'_i, L'_i, R'_i)$ for each $i = 1, \ldots, p$.
 All of these rooted configurations are non-trivial and $(A, L'_1, R'_p) = H'_1 \times \cdots \times H'_p$.
 Recall that $H'_1$ is spanning because it is a product whose first term is $H_1$, which is spanning. 
 The fact that $|A'_i| \leq 2^{(|\mathbb{F}|^w+1)i}$ for each $i = 1, \ldots, p-1$ follows from the fact that $|A_1 \cup \cdots \cup A_{\ell_i}| \leq 2^{\ell_i}$ and $\ell_i \leq (|\mathbb{F}|^w + 1)i$.
 
 By \autoref{lem:decomposeintoproduct0} we have $\kappa_{M(A)}(A'_1 \cup \cdots \cup A'_i, A'_j \cup \cdots \cup A'_p) \geq \min\{|R_{\ell_i}|, |R_{\ell_i+1}|, \ldots, |R_{\ell_{j-1}}|\}$ for any $i < j$.
 But recall that each $R_n$ spans the boundary of $A_1 \cup \cdots \cup A_n$ in $A$ so $|R_n| = |R_{n-1}|$ for all $n > 1$ such that $H_n$ is trivial. 
 Hence $\min\{|R_{\ell_i}|, |R_{\ell_i+1}|, \ldots, |R_{\ell_{j-1}}|\} = \min\{|R_{\ell_k}| : k = i, \ldots, j-1\}$.
 So $\kappa_{M(A)}(A'_1 \cup \cdots \cup A'_i, A'_j \cup \cdots \cup A'_p) \geq \min\{|R'_k| : k = i, \ldots, j-1\}$.
\end{proof}

\section{Rational limiting densities} \label{sec:rational}

For the remainder of the paper, we let $\mathbb{F}$ denote a finite field.
In this section, we prove that the limiting density of any minor-closed class of $\mathbb{F}$-representable matroids of bounded branch-width is a rational number. First we prove the following structural theorem, and afterwards we will combine it with well-quasi-ordering to get this result.
We call a sequence of configurations pruned if the corresponding sequence of matroids is.

\begin{theorem} \label{thm:boundedbranchwidth}
 Let $w$ be an integer, let $\mathcal{F}$ be a minor-closed class of configurations over $\mathbb{F}$ with limiting density $\Delta$, and let $\{A_i : i \geq 1\}$ be a pruned sequence of configurations in $\mathcal{F}$ with branch-width at most $w$ such that $d(A_i) \rightarrow \Delta$ and $\epsilon(A_i) \rightarrow \infty$.
 There is an integer $q$ and an infinite sequence of non-trivial linked $q$-patches $\{H_j = (\widetilde{H_j}, S_j, T_j): j \geq 1\}$ such that, for each $j = 1, 2, \ldots$,
 \begin{enumerate}[(i)]
  \item $\widetilde{H_j} \cap \spn{S_j}$ is empty, \label{con:first}
  \item $\epsilon(\widetilde{H_j}) \geq \Delta(\dim(\spn{H_j}) - q)$, and \label{con:third}
  \item there is a rooted configuration $F_j$ in $\mathcal{P}(H_1, \ldots, H_j)$ such that $\widetilde{F_j} \in \mathcal{F}$. \label{con:second}
 \end{enumerate}
\end{theorem}

\begin{proof}
We may assume each $M(A_i)$ is simple.
Since $\epsilon(A_i) \rightarrow \infty$, for each positive integer $p$ there is a configuration $A_{m(p)}$ with $\epsilon(A_{m(p)}) > 2^{(|\mathbb{F}|^w + 1)p}$.
By replacing our sequence of configurations with this subsequence we may assume that $\epsilon(A_p) > 2^{(|\mathbb{F}|^w + 1)p}$ for all positive integers $p$.

Hence by \autoref{lem:decomposeintoproduct}, for each $p \geq 1$ there is a rooted configuration $(A_p, L_{p, 1}, R_{p, p})$ and $p$ non-trivial $(\leq w)$-rooted configurations $E_{p, 1} = (B_{p, 1}, L_{p, 1}, R_{p, 1}), \ldots, E_{p, p} = (B_{p, p}, L_{p, p}, R_{p, p})$, such that
\[ (A_p, L_{p, 1}, R_{p, p}) = E_{p, 1} \times \cdots \times E_{p, p}, \]
$E_{p, 1}$ is spanning, $|B_{p, i}| \leq 2^{(|\mathbb{F}|^w+1)i}$ for all $i = 1, \ldots, p-1$, and $\kappa_{M(A_p)}(B_{p, 1} \cup \cdots \cup B_{p, k}, B_{p, \ell} \cup \cdots \cup B_{p, p}) \geq \min\{|R_{p, k}|, |R_{p, k+1}|, \ldots, |R_{p, \ell-1}|\}$ for any $k < \ell$.
We may assume that $B_{p, i} \cap \spn{L_{p, 1}}$ is empty for each $i = 2, \ldots, p$ by moving any element $e$ of this set into $B_{p, k}$ for the smallest possible $k$ where $e \in \spn{R_{p, k}}$.

For each fixed positive integer $j$, the sets $\{B_{p, j} : p > j\}$ all have size at most $2^{(|\mathbb{F}|^w+1)j}$. Hence, for each integer $i$, the rooted configurations $\{E_{p, 1} \times \cdots \times E_{p, i} : p \geq 1 \}$ fall into finitely many isomorphism classes.

In particular, there are infinitely many values of $p$ such that the rooted configurations $E_{p, 1}$ are all isomorphic to each other. Let $p(1)$ be one such value of $p$.
We define a sequence $\{p(i) : i \geq 1\}$ inductively; fix $i$ and suppose $p(i-1)$ is defined. There are infinitely many values of $p$ such that the rooted configurations $E_{p, 1} \times \cdots \times E_{p, i}$ are all isomorphic to each other and such that the rooted configurations $E_{p, 1} \times \cdots \times E_{p, i-1}$ are all isomorphic to $E_{p(i-1), 1} \times \cdots \times E_{p(i-1), i-1}$; let $p(i)$ be such a value of $p$.

So for any natural numbers $i$ and $j$ with $i < j$, the configuration $B_{p(i), 1} \cup \cdots \cup B_{p(i), i}$ is isomorphic to $B_{p(j), 1} \cup \cdots \cup B_{p(j), i}$.

Set $q = \liminf_{i \rightarrow \infty} |R_{p(i), i}|$.
Then there is an infinite sequence $i_1, i_2, \ldots$ such that 
\begin{enumerate}[(a)]
 \item $|R_{p(k), k}| \geq q$ for all $k \geq i_1$, and \label{item:first}
 \item $|R_{p(i_j), i_j}| = q$ for all $j \geq 1$.
\end{enumerate}

We define a sequence of $(\leq w)$-rooted configurations $\{H'_j : j \geq 1\}$ as follows.
For each $j \geq 1$, we set
\[ H'_j = E_{p(i_{j+1}), i_j + 1} \times \cdots \times E_{p(i_{j+1}), i_{j+1}} .\]
Each rooted configuration $H'_j$ is a non-trivial $q$-patch.

For each $j \geq 1$, we will turn $H'_j$ into a linked patch $H_j$ by re-defining its terminals.
First, we define $X_1 = R_{p(i_2), i_2}$ and we set $H_1 = (\widetilde{H'_1}, X_1, X_1)$. So $H_1$ is a linked patch.
Now, suppose that we have defined the patches $H_1, \ldots, H_{j-1}$. We define $H_j$ inductively as follows.
Let $U_j = B_{p(i_{j+1}), 1} \cup \cdots \cup B_{p(i_{j+1}), i_j}$ and let $V_j = B_{p(i_{j+1}), i_{j+1} + 1} \cup \cdots \cup B_{p(i_{j+1}), p(i_{j+1})}$.
Then $H'_j = (A_{p(i_{j+1})} - U_j - V_j, R_{p(i_{j+1}), i_j}, R_{p(i_{j+1}), i_{j+1}})$.
Let $M = M(A_{p(i_{j+1})})$.
It follows from (\ref{item:first}) that $\kappa_M(U_j, V_j) \geq q$.
Therefore, by Tutte's Linking Theorem, 
there is a partition $(C, D)$ of $A_{p(i_{j+1})} - U_j - V_j$ such that $\lambda_{M / C \backslash D}(U_j) \geq q$.

Thus, there is a linear transformation $\mathcal{L}$ on $\spn{A_{p(i_{j+1})}}$ with $\ker(\mathcal{L}) = \spn{C}$ such that $\spn{\mathcal{L}(U_j)} \cap \spn{\mathcal{L}(V_j)} \geq q$. Since the right boundary of $U_j$ is contained in $\spn{R_{p(i_{j+1}), i_j}}$ and the left boundary of $V_j$ is contained in $\spn{L_{p(i_{j+1}), i_{j+1} + 1}} = \spn{R_{p(i_{j+1}), i_{j+1}}}$, both have dimension at most $q$.
This means that the boundaries of $U_j$ and $V_j$ have the same image under $\mathcal{L}$.

Let $X_{j-1}$ be the set of right terminals of $H_{j-1}$ and call its elements $X_{j-1} = \{x_1, \ldots, x_q\}$. 
Then we can define an ordered basis $X_j = \{x'_1, \ldots, x'_q\}$ of the boundary of $V_j$ by setting each $x'_i$ to be the element of this boundary such that $\mathcal{L}(x'_i) = \mathcal{L}(x_i)$.
We set $H_j = (A_{p(i_{j+1})} - U_j - V_j, X_{j-1}, X_j)$. Then $H_j$ has the minor $(\mathcal{L}(A_{p(i_{j+1})} - U_j - V_j), \mathcal{L}(X_{j-1}), \mathcal{L}(X_j))$, so it is a linked patch.

The sequence $\{H_j : j \geq 1\}$ satisfies (\ref{con:first}) and (\ref{con:second}). It remains to show that (\ref{con:third}) holds.

We fix some $j \geq 1$.
Set $k = \dim(\spn{H_j}) - q$.
Let $\delta$ be any positive real number.
Since $\{A_i : i \geq 1\}$ is a pruned sequence of configurations, there is an integer $N$ such that $A_i$ is $(\delta, k)$-pruned for all $i \geq N$.
Recall that there are infinitely many values of $p$ for which the configuration $A_p$ is equal to $\widetilde{J_p}$ for a rooted configuration $J_p \in \mathcal{P}(H_1, \ldots, H_j, \ldots, H_{n(p)})$ for some $n(p) \geq j$.
We may thus choose one such $A_\ell$ such that $d(A_\ell) > \Delta - \delta$ and $\ell \geq N$; so $A_\ell$ is $(\delta, k)$-pruned.
So $\widetilde{H_j}$ is isomorphic to a subconfiguration of $A_\ell$; we identify this subconfiguration with $\widetilde{H_j}$ itself.

We can write $A_\ell$ as $\widetilde{J_\ell}$ where $J_\ell \in \mathcal{P}(G_1, H_j, G_2)$ for some two rooted configurations $G_1$ in $\mathcal{P}(H_1, \ldots, H_{j-1})$ and $G_2$ in $\mathcal{P}(H_{j+1}, \ldots, H_{n(\ell)})$ for some $n(\ell) \geq j$. Since $E_{p(1), 1}$ is spanning, so is $G_1$.
Since $H_j$ is a linked $q$-patch, $J_\ell$ has a minor $J'$ in $\mathcal{P}(G_1, G_2)$.
We observe that $\epsilon(A_\ell) - \epsilon(\widetilde{J'}) = \epsilon(\widetilde{H_j})$. Also, $\dim(\spn{J_\ell}) - \dim(\spn{J'}) = \dim(\spn{H_j}) - q = k$. Since $G_1$ is spanning, so are $J_\ell$ and $J'$, so $\dim(\spn{A_\ell}) - \dim(\spn{\widetilde{J'}}) = k$.
Therefore, the fact that $A_\ell$ is $(\delta, k)$-pruned means that
\begin{align*}
\epsilon(A_\ell) - \epsilon(\widetilde{J'}) &\geq (d(A_\ell) - \delta) (\dim(\spn{A_\ell}) - \dim(\spn{\widetilde{J'}}) \\
\epsilon(\widetilde{H_j})                   &\geq (d(A_\ell) - \delta) (\dim(\spn{H_j}) - q) \\
                                            &> (\Delta - 2\delta)(\dim(\spn{H_j}) - q).
\end{align*}
Since this is true for arbitrary $\delta$, the theorem follows.
\end{proof}

The next theorem implies that the limiting density of any minor-closed class of $\mathbb{F}$-representable matroids of bounded branch-width is rational.

\begin{theorem} \label{thm:longchains}
 Let $w$ be an integer and let $\mathcal{F}$ be a minor-closed class of configurations over $\mathbb{F}$ of branch-width at most $w$ with limiting density $\Delta > 0$.
 There is an integer $q$ and a non-trivial linked $q$-patch $H = (\widetilde{H}, L, R)$ such that
 \begin{enumerate}[(i)]
  \item $\widetilde{H} \cap \spn{L}$ is empty,
  \item $\epsilon(\widetilde{H}) = \Delta(\dim(\spn{H}) - q)$, and
  \item there is a rooted configuration $F_n$ in $\mathcal{P}(H^n)$ such that $\widetilde{F_n} \in \mathcal{F}$, for every $n \geq 1$.
 \end{enumerate}
\end{theorem}

\begin{proof}
 By \autoref{lem:prunedsequenceexists} there is a pruned sequence $\{A_i : i \geq 1\}$ of configurations in $\mathcal{F}$ such that $d(A_i) \rightarrow \Delta$ and $\dim(A_i) \rightarrow \infty$.
 Then \autoref{thm:boundedbranchwidth} applies; we let $\{H_j : j \geq 1\}$ be the sequence of $q$-patches it gives.
 It follows from \autoref{thm:rootedwqo} and the properties of well-quasi-orders that $\{H_j : j \geq 1\}$ contains an infinite subsequence $H_{i_1}, H_{i_2}, H_{i_3}, \ldots$ such that $H_{i_j}$ is a minor of $H_{i_k}$ for all $j < k$.
 We set $H = H_{i_1}$. 
 Recall that $\epsilon(\widetilde{H}) \geq \Delta(\dim(\spn{H}) - q)$.
 
 We know from \autoref{thm:boundedbranchwidth} that there is a rooted configuration $F_n$ in $\mathcal{P}(H_1, H_2, \ldots, H_{i_n})$ such that $\widetilde{F_n} \in \mathcal{F}$ for each $n$. 
 Since all the $q$-patches $H_j$ are linked, there is a minor $F'_n$ of $F_n$ that is in $\mathcal{P}(H_{i_1}, H_{i_2}, \ldots, H_{i_n})$.
 Since $H = H_{i_1}$ is a minor of each of $H_{i_2}, \ldots, H_{i_n}$, there is also a minor $F''_n$ of $F'_n$ that is in $\mathcal{P}(H^n)$. Note that $\widetilde{F''_n}$ is a minor of $\widetilde{F_n}$ so it is in $\mathcal{F}$.
 Since $H$ is non-trivial, $\dim(\spn{\widetilde{F''_n}}) \rightarrow \infty$ so we have $\Delta \geq \limsup_{n \rightarrow \infty} d(\widetilde{F''_n})$.
 Since $\dim(\spn{\widetilde{H}})$ and $\dim(\spn{H})$ differ by at most $q$, and $\dim(\spn{F''_n}) = q + n(\dim(\spn{H}) - q)$, we have
 \[ \frac{n \epsilon(\widetilde{H})}{q + n(\dim(\spn{H}) - q)} \leq d(\widetilde{F''_n}) \leq \frac{n \epsilon(\widetilde{H})}{n(\dim(\spn{H}) - q)} \]
 and hence $\lim_{n \rightarrow \infty} d(\widetilde{F''_n}) = \frac{\epsilon(\widetilde{H})}{\dim(\spn{H}) - q} \leq \Delta$.
 Therefore, $\epsilon(\widetilde{H}) = \Delta(\dim(\spn{H}) - q)$.
\end{proof}

The second conclusion of this theorem has the following consequence.

\begin{corollary} \label{cor:rational}
 For each finite field $\mathbb{F}$ and each minor-closed class $\mathcal{M}$ of $\mathbb{F}$-representable matroids of bounded branch-width, the limiting density of $\mathcal{M}$ is a rational number.
\end{corollary}

\section{The extremal function}

In this section, we characterize the extremal functions of all minor-closed classes of matroids of bounded branch-width representable over a finite field $\mathbb{F}$.
We define the notation $\mathcal{P}(G_1, H^K, G_2)$ to signify the set $\mathcal{P}(G_1, H, \ldots, H, G_2)$, where $H$ appears $K$ times. The next theorem provides conditions under which we can find elements of a minor-closed class belonging to such sets for arbitrarily large values of $K$. Later, we will show that extremal matroids come from rooted configurations having this form.

\begin{theorem} \label{thm:expand}
 For any natural number $q$, any minor-closed class $\mathcal{F}$ of $q$-patches over $\mathbb{F}$ of bounded branch-width, and any linked $q$-patch $H$ in $\mathcal{F}$, there is an integer $K = K_{\ref{thm:expand}}(H, \mathcal{F})$ such that for all $q$-patches $G_1$ and $G_2$ in $\mathcal{F}$, if $\mathcal{F}$ contains an element of $\mathcal{P}(G_1, H^{K'}, G_2)$ for some $K' \geq K$, then $\mathcal{F}$ contains an element of $\mathcal{P}(G_1, H^L, G_2)$ for all $L \geq 0$.
\end{theorem}

\begin{proof}
 Consider the set $\mathcal{Q} = \mathcal{F} \times \mathbb{N} \times \mathcal{F}$ along with the relation $\leq_\mathcal{Q}$ defined by setting $(G'_1, k', G'_2) \leq_{\mathcal{Q}} (G_1, k, G_2)$ if and only if $G'_1$ is a minor of $G_1$, $k' \leq k$ and $G'_2$ is a minor of $G_2$.
 Both $\mathcal{F}$ and $\mathbb{N}$ are well-quasi-orders (under the minor relation and the $\leq$ relation, respectively) and the Cartesian product of two well-quasi-orders is one as well, so $\mathcal{Q}$ is well-quasi-ordered by $\leq_\mathcal{Q}$.
 
 Define the set $\hat{\mathcal{F}} \subseteq \mathcal{Q}$ to be the downward closure under $\leq_\mathcal{Q}$ of the set of all triples $(G_1, k, G_2)$ with the property that $\mathcal{F}$ contains an element of $\mathcal{P}(G_1, H^k, G_2)$.
 Since $\mathcal{Q}$ is a well-quasi-order, there is a finite set $\mathcal{O} \subset \mathcal{Q}$ consisting of the $\leq_\mathcal{Q}$-minimal elements not in $\hat{\mathcal{F}}$. We pick an integer $K > \max\{k : (G_1, k, G_2) \in \mathcal{O}\}$.
 
 Suppose there are $q$-patches $G_1$ and $G_2$ in $\mathcal{F}$ such that $\mathcal{F}$ contains an element of $\mathcal{P}(G_1, H^{K'}, G_2)$ for some $K' \geq K$ but not any element of $\mathcal{P}(G_1, H^L, G_2)$ for some $L \geq 0$.
 
 If $(G_1, L, G_2) \in \hat{\mathcal{F}}$, then there is a triple $(G'_1, L', G'_2)$ such that $(G_1, L, G_2) \leq_\mathcal{Q} (G'_1, L', G'_2)$ and $\mathcal{F}$ contains an element of $\mathcal{P}(G'_1, H^{L'}, G'_2)$.
 But then, since $G_1$ is a minor of $G'_1$ and $G_2$ is a minor of $G'_2$, it follows that $\mathcal{F}$ contains an element of $\mathcal{P}(G_1, H^{L'}, G_2)$. Moreover, since $H$ is linked, it follows from \autoref{prop:contractinglinkedpatch} that $\mathcal{F}$ contains an element of $\mathcal{P}(G_1, H^L, G_2)$, a contradiction. This proves that $(G_1, L, G_2) \not\in \hat{\mathcal{F}}$.
 
 There therefore exists an element $(H_1, N, H_2) \in \mathcal{O}$ such that $(H_1, N, H_2) \leq_\mathcal{Q} (G_1, L, G_2)$. 
 It then follows that $H_1$ is a minor of $G_1$ and $H_2$ is a minor of $G_2$ and so, since $K' \geq K > N$, we have $(H_1, N, H_2) \leq_\mathcal{Q} (G_1, K', G_2)$. 
 Since $\hat{\mathcal{F}}$ is downwardly-closed under the $\leq_\mathcal{Q}$ relation, this means that $(G_1, K', G_2) \not\in \hat{\mathcal{F}}$, which contradicts the fact that $\mathcal{F}$ contains an element of $\mathcal{P}(G_1, H^{K'}, G_2)$.
\end{proof}

\subsection*{Decomposing into linked patches}

Here we show that a large enough configuration of bounded branch-width can be decomposed into a product of linked $q$-patches for some integer $q$.

\begin{lemma} \label{lem:getpathdecomposition}
 For any positive integers $p$ and $w$ and configuration $A$ over $\mathbb{F}$ of branch-width at most $w$ with $\epsilon(A) \geq 2^{(|\mathbb{F}|^w+1)p^{w+1}}$, there is an integer $q$ such that $0 \leq q \leq w$ and a $q$-patch $H$ such that $\widetilde{H} = A$ and $H$ is a product of $p$ non-trivial linked $q$-patches $H_1 \times \cdots \times H_p$ where $H_1$ is spanning.
\end{lemma}

\begin{proof}
 Let $A$ be a configuration over $\mathbb{F}$ of branch-width at most $w$ with $\epsilon(A) \geq 2^{(|\mathbb{F}|^w+1)p^{w+1}}$.
 By \autoref{lem:decomposeintoproduct}, there is a product of $p^{w+1}$ non-trivial $(\leq w)$-rooted configurations:
 \[ (A, L_1, R_{p^{w+1}}) = (A_1, L_1, R_1) \times \cdots \times (A_{p^{w+1}}, L_{p^{w+1}}, R_{p^{w+1}}) \]
 such that $(A_1, L_1, R_1)$ is spanning and $\kappa_{M(A)}(A_1 \cup \cdots \cup A_i, A_j \cup \cdots \cup A_{p^{w+1}}) \geq \min\{|R_i|, |R_{i+1}|, \ldots, |R_{j+1}|\}$ for any $i < j$.
 
 \begin{claim} \label{clm:findq}
  There is an integer $q$ and there are $p$ indices $j_1 < \cdots < j_p$ such that $|R_{j_1}|, |R_{j_2}|, \ldots, |R_{j_p}|$ are all equal to $q$ and $|R_i| \geq q$ whenever $j_1 \leq i \leq j_p$.
 \end{claim}

 Let $q$ be the maximum integer such that there exists an integer $k$ with $|R_{k+1}|, \ldots, |R_{k+p^{w-q+1}}| \geq q$; such $q$ exists because these inequalities hold when $q = 0$ and $k = 0$.
 If fewer than $p$ of the numbers $|R_{k+1}|, \ldots, |R_{k+p^{w-q+1}}|$ are equal to $q$, then some stretch of at least $(p^{w-q+1} - (p-1)) / p > p^{w-q} - 1$ of them are greater than $q$. That is, there is a $k'$ such that $|R_{k'+1}|, \ldots, |R_{k'+p^{w-q}}| \geq q + 1$, contradicting the maximality of $q$.
 Hence we can choose the $p$ indices $j_1, \ldots, j_p$ in the set $\{k+1, \ldots, k+p^{w-q+1}\}$. 
 This proves (\ref{clm:findq}).
 \\

 Let $q$ and $j_1, \ldots, j_p$ be as given by (\ref{clm:findq}).
 We define $H'_1 = (A_1 \cup \cdots \cup A_{j_1}, R_{j_1}, R_{j_1})$. This is spanning because $(A_1, L_1, R_1)$ is.
 For each $i = 2, \ldots, p-1$, we set $H'_i = (A_{j_{i-1} + 1} \cup \cdots \cup A_{j_i}, L_{j_{i-1} + 1}, R_{j_i})$.
 Finally, we let $H'_p = (A_{j_{p-1} + 1} \cup \cdots \cup A_{p^{w+1}}, L_{j_{p-1} + 1}, L_{j_{p-1} + 1})$.
 These are all $q$-patches since each $L_{j_{i-1} + 1}$ is equal to $R_{j_{i-1}}$.
 
 Then $(A, R_{j_1}, L_{j_{p-1} + 1}) = H'_1 \times \cdots \times H'_p$. All the $q$-patches in this product are non-trivial because each of the rooted configurations $(A_k, L_k, R_k)$ is non-trivial.
 Next, we modify the terminals of these patches to make sure they are linked.
 Since its left and right terminals are the same, $H'_1$ is linked. We set $H_1 = H'_1$ (so $H_1$ is spanning) and let $X_1 = R_{j_1}$.
 We inductively define $X_2, \ldots, X_p$ as follows.
 Let $k \in \{2, \ldots, p-1\}$ and suppose that $X_1, \ldots, X_{k-1}$ have been defined to be bases of the spaces $\spn{R_{j_1}}, \ldots, \spn{R_{j_{k-1}}}$.
 
 We have $\kappa_{M(A)}(A_1 \cup \cdots \cup A_{j_{k-1}}, A_{j_k+1} \cup \cdots \cup A_{p^{w+1}}) \geq q$.
 This means that there is a linear transformation $\mathcal{L}_j$ on $\spn{A}$ whose kernel is the span of a subset of $A_{j_{k-1} + 1} \cup \cdots \cup A_{j_k}$ and such that $\mathcal{L}_j(\spn{L_{j_{k-1} + 1}}) = \mathcal{L}(\spn{R_{j_k}})$ and this space has dimension $q$.
 Moreover, $X_{k-1}$ is a basis of $\spn{L_{j_{k-1} + 1}}$ so if we set $X_k = \mathcal{L}^{-1}(\mathcal{L}(X_{k-1})) \cap \spn{R_{j_k}}$, then $X_k$ is a basis of $\spn{R_{j_k}}$. Choosing the appropriate ordering of the elements of $X_k$, we see that
 \[ H_k = (A_{j_{k-1} + 1} \cup \cdots \cup A_{j_k}, X_{k-1}, X_k) \]
 is a linked $q$-patch.
 
 Finally, we can define $H_p = (A_{i_{p-1} + 1} \cup \cdots \cup A_{p^{w+1}}, X_{p-1}, X_{p-1})$, which is also a linked $q$-patch.
 So we have $(A, X_1, X_{p-1}) = H_1 \times \cdots \times H_p$.
\end{proof}

\subsection*{Bounding the extremal size}

Next, we show that for every minor-closed class $\mathcal{F}$ of configurations of bounded branch-width with limiting density $\Delta$, there is a constant bound on $|ex_\mathcal{F}(n) - \Delta n|$.

\begin{lemma} \label{lem:bounded}
 For any minor-closed class $\mathcal{F}$ of configurations of bounded branch-width over a finite field $\mathbb{F}$ with limiting density $\Delta$, there is a number $c_{\ref{lem:bounded}}(\mathcal{F})$ so that $|ex_\mathcal{F}(n) - \Delta n| < c_{\ref{lem:bounded}}(\mathcal{F})$ for all $n \geq 1$.
\end{lemma}

\begin{proof}
 For each configuration $A$ in $\mathcal{F}$, we define $f(A) = \epsilon(A) - \Delta \dim(\spn{A})$, so for each positive integer $n$, we have $ex_\mathcal{F}(n) - \Delta n = \max \{ f(A) : A \in \mathcal{F}, \dim(\spn{A}) = n \}$.

 First, we prove that $ex_\mathcal{F}(n) - \Delta n$ is bounded below.
 By \autoref{thm:longchains}, there is an integer $q$ and a non-trivial $q$-patch $H = (\widetilde{H}, L, R)$ such that $\widetilde{H} \cap \spn{L}$ is empty, $\epsilon(\widetilde{H}) = \Delta(\dim(\spn{H}) - q)$, and for all $k \geq 1$ there is a rooted configuration $F_k \in \mathcal{P}(H^k)$ such that $\widetilde{F_k} \in \mathcal{F}$.
 
 Note that for any $F$ in $\mathcal{P}(H^k)$, we have $\dim(\spn{\widetilde{F}}) \leq \dim(\spn{F})$ and $\dim(\spn{F}) = q + k(\dim(\spn{H}) - q)$, so
 \[ f(\widetilde{F}) \geq k\epsilon(\widetilde{H}) - \Delta (q + k (\dim(\spn{H}) - q)) = -\Delta q. \]
 We observe that for any two elements $F, F'$ of $\mathcal{P}(H^k)$, we have $|\dim(\spn{\widetilde{F}}) - \dim(\spn{\widetilde{F'}})| \leq q$.
 
 Fix some $n \geq 1$. We let $k$ be the smallest integer such that $\dim(\spn{\widetilde{F_k}}) \geq n$. 
 For each $j < k$, let $F'_j$ be an element of $\mathcal{P}(H^j)$ such that $\widetilde{F'_j}$ is a subconfiguration of $\widetilde{F_k}$.
 Since $n > \dim(\spn{\widetilde{F_{k-1}}})$, we have $n > \dim(\spn{\widetilde{F'_{k-1}}}) - q$.
 Hence $n > \dim(\spn{\widetilde{F'_{k-1-q}}})$ because $H$ is non-trivial.
 So there exists a configuration $A$ in $\mathcal{F}$ with $\dim(\spn{A}) = n$ such that $A$ is a subconfiguration of $\widetilde{F_k}$ and $\widetilde{F'_{k-1-q}}$ is a subconfiguration of $A$.
 So
 \begin{align*}
  ex_\mathcal{F}(n) - \Delta n &\geq \epsilon(A) - \Delta n \\
                               &\geq \epsilon(\widetilde{F'_{k-1-q}}) - \Delta n\\
                               &= f(\widetilde{F'_{k-1-q}}) - \Delta \left(n - \dim(\spn{\widetilde{F'_{k-1-q}}}) \right).
 \end{align*}
 However, $n \leq \dim(\spn{\widetilde{F_k}}) \leq \dim(\spn{\widetilde{F'_{k-1-q}}}) + (q+1)\dim(\spn{H})$.
 So
 \begin{align*}
  ex_\mathcal{F}(n) - \Delta n &\geq f(\widetilde{F'_{k-1-q}}) - \Delta(q+1)\dim(\spn{H}) \\
                               &\geq -\Delta q - \Delta(q+1)\dim(\spn{H}),
 \end{align*}
 which proves that $ex_\mathcal{F}(n) - \Delta n$ is bounded from below by a constant depending only on the class $\mathcal{F}$.
 
 Next, we show that $ex_\mathcal{F}(n) - \Delta n$ is bounded above.
 We assume that it is not.
 There is then a sequence of configurations $\{G_i : i \geq 1\}$ in $\mathcal{F}$ such that $\dim(\spn{G_i}) \rightarrow \infty$ and $f(G_i) \rightarrow \infty$. We may assume that, for each $i$, every proper minor $G'$ of $G_i$ satisfies $f(G') < f(G_i)$.
 
 By \autoref{lem:getpathdecomposition}, for each positive integer $n$ there is an integer $q(n)$ and a configuration $G_{i(n)}$ in this sequence such that there is a $q(n)$-patch $(G_{i(n)}, L, R)$ that is a product of $n$ non-trivial linked $q(n)$-patches, the first of which is spanning.
 
 Some value appears infinitely among the $q(n)$; call it $q$. We may then assume that $q(n) = q$ for all $n$ (we take the subsequence of configurations with this value of $q(n)$ and for each $n$ we take one that is a product of $n' \geq n$ $q$-patches and group the $n'$ $q$-patches into $n$ of them).
 For each $n$, we have a product $(G_{i(n)}, L, R) = H_{n, 1} \times \cdots \times H_{n, n}$ where each $H_{n, i}$ is a non-trivial linked $q$-patch and $H_{n, 1}$ is spanning.
 Let each $H_{n, i} = (\widetilde{H_{n, i}}, L_{n, i}, R_{n, i})$.
 We may assume that, when $k \geq 2$, the set $\widetilde{H_{n, k}} \cap \spn{L_{n, k}}$ is empty, by moving each member $e$ of this set into the $q$-patch $H_{n, j}$ for the smallest $j$ such that $e \in \spn{R_{n, j}}$.
 
 For each $n$ and $k$, there is an element $J$ of $\mathcal{P}(H_{n, 1}, \ldots, H_{n, k-1}, H_{n, k+1}, \ldots, H_{n, n})$ that is a minor of $(G_{i(n)}, L, R)$, because $H_{n, k}$ is a linked $q$-patch, by \autoref{prop:contractinglinkedpatch}.
 We have $\epsilon(\widetilde{H_{n, k}}) = \epsilon(\widetilde{G_{i(n)}}) - \epsilon(\widetilde{J})$.
 Therefore, the fact that $f(G_{i(n)}) > f(J)$ means that $\epsilon(\widetilde{H_{n, k}}) > \Delta(\dim(\spn{\widetilde{G_{i(n)}}}) - \dim(\spn{\widetilde{J}}))$.
 Since $H_{n, 1}$ is spanning, so are $G_{i(n)}$ and $J$, so $\epsilon(\widetilde{H_{n, k}}) > \Delta(\dim(\spn{G_{i(n)}}) - \dim(\spn{J})) = \Delta(\dim(\spn{H_{n, k}}) - q)$.
 
 Let $\mathcal{G}$ be the set of all non-trivial linked $q$-patches $H$ such that $\epsilon(\widetilde{H}) > \Delta(\dim(\spn{H}) - q)$ and $\widetilde{H} \in \mathcal{F}$. 
 So all the patches $H_{n, k}$ are in $\mathcal{G}$.
 Since any set of $q$-patches over $\mathbb{F}$ of bounded branch-width is well-quasi-ordered by minors, the set of minor-minimal elements of $\mathcal{G}$ is finite; call it $\mathcal{O}$.
 Define 
 \[ \delta = \min \{ \epsilon(\widetilde{H}) - \Delta(\dim(\spn{H}) - q) : H \in \mathcal{O} \} \]
 and
 \[ m = \max \{ \dim(\spn{H}) - q : H \in \mathcal{O} \}. \]

 The fact that $\mathcal{O}$ is finite means that these numbers are well-defined; we have $\delta > 0$ and $m > 0$ by the definition of $\mathcal{G}$.
 For each $n$ and each $k$, the $q$-patch $H_{n, k}$ has a minor $H'_{n, k}$ in $\mathcal{O}$. 
 For each $n$, $(G_{i(n)}, L, R)$ has a minor $P_n$ which is in $\mathcal{P}(H'_{n, 1}, \ldots, H'_{n, n})$. Also, $\widetilde{P_n} \in \mathcal{F}$ since it is a minor of $G_{i(n)}$.
 We have
 \begin{align*}
   d(\widetilde{P_n})
          &= \frac{\epsilon(\widetilde{P_n})}{\dim(\spn{\widetilde{P_n}})} = \frac{\sum_{i=1}^n \epsilon(\widetilde{H'_{n,i}})}{\dim(\spn{\widetilde{P_n}})} \\
          &\geq \frac{\sum_{i=1}^n (\delta + \Delta(\dim(\spn{H'_{n,i}}) - q))}{\dim(\spn{P_n})} \\
          &= \frac{n\delta + \Delta \sum_{i=1}^n (\dim(\spn{H'_{n,i}}) - q)}{\dim(\spn{P_n})} \\
          &= \frac{n\delta + \Delta (\dim(\spn{P_n}) - q)}{\dim(\spn{P_n})} \\
          &= \Delta + \frac{n \delta - \Delta q}{q + \sum_{i=1}^n (\dim(\spn{H'_{n,i}}) - q)} \\
          &\geq \Delta + \frac{n \delta - \Delta q}{q + n m} = \Delta + \frac{\delta - \frac{\Delta q}{n}}{\frac{q}{n} + m}.
 \end{align*}
 So $\limsup_{n \rightarrow \infty} d(\widetilde{P_n}) \geq \Delta + \frac{\delta}{m}$, which is a contradiction because $\delta/m > 0$ and $\Delta$ is the limiting density of $\mathcal{F}$.
\end{proof}

\subsection*{Characterizing the extremal configurations}

We can almost prove our main result, but need one short technical lemma.

\begin{lemma} \label{lem:getzerosubsequence}
 Let $k, P$, and $N$ be integers. 
 If $N \geq kP$ and $a_1, \ldots, a_N$ is a sequence of $N$ integers, then there are integers $m$ and $\ell$ so that $\ell \geq k$ and $\sum_{i=m+1}^{m+\ell} a_i \equiv 0 \pmod P$.
\end{lemma}

\begin{proof}
 Let $b_0, b_1, \ldots, b_N$ be the sequence of partial sums; that is $b_j = \sum_{i = 0}^j a_i$ for all $j = 0, \ldots, N$.
 It suffices to show that there are $m$ and $\ell$ so that $\ell \geq k$ and $b_m \equiv b_{m+\ell} \pmod P$.
 
 For each $v \in \{0, \ldots, P-1\}$, let $i(v)$ and $j(v)$ be the minimum and maximum indices such that $b_{i(v)} \equiv v \pmod P$ and $b_{i(v)} \equiv v \pmod P$. If $j(v) - i(v) < k$ for all $v$, then it follows that $N + 1 < kP$, a contradiction. So for some $v$, we have $j(v) - i(v) \geq k$. We set $m = i(v)$ and $\ell = j(v) - i(v)$.
\end{proof}

Finally, we prove our main structural theorem, which will imply \autoref{thm:maintheorem}.

\begin{theorem} \label{thm:path}
 For each minor-closed class $\mathcal{F}$ of configurations of bounded branch-width over a finite field $\mathbb{F}$, there are integers $P$ and $M$ such that the following holds.
 For each integer $i$, there is an integer $q$ and $q$-patches $G_1, H, G_2$ such that whenever $n \equiv i \pmod P$ and $n > M$, there is a spanning $q$-patch $F$ in $\mathcal{P}(G_1, H^L, G_2)$ for some $L$ such that $\widetilde{F} \in \mathcal{F}$, $\dim(\spn{\widetilde{F}}) = n$, and $\epsilon(\widetilde{F}) = ex_\mathcal{F}(n)$.
\end{theorem}

\begin{proof}
 Let $w$ be the maximum branch-width of configurations in $\mathcal{F}$ and let $\Delta$ be the limiting density of $\mathcal{F}$.
 We define $f(G) = \epsilon(G) - \Delta \dim(\spn{G})$ for each configuration $G$, so $ex_\mathcal{F}(n) - \Delta n = \max\{ f(G) : G \in \mathcal{F}, \dim(\spn{G}) = n \}$.
 For a rooted configuration $H$ and number $q$, we define $g_q(H) = \epsilon(\widetilde{H}) - \Delta(\dim(\spn{H}) - q)$.
 
 \begin{claim} \label{clm:computef}
 If $J, G$ and $H$ are $q$-patches such that $J \in \mathcal{P}(G, H)$ and $G$ is a spanning patch, then $f(\widetilde{J}) = f(\widetilde{G}) + g_q(H)$.
 \end{claim}
 
 We have 
 \begin{align*}
 \dim(\spn{J}) &= \dim(\spn{G}) + \dim(\spn{H}) - q \\
               &= \dim(\spn{\widetilde{G}}) + \dim(\spn{H}) - q \\
 \epsilon(\widetilde{J}) - \Delta \dim(\spn{J}) &= \epsilon(\widetilde{G}) - \Delta\dim(\spn{\widetilde{G}}) + \epsilon(\widetilde{H}) - \Delta(\dim(\spn{H}) - q),
 \end{align*}
 where the last line follows because $\epsilon(\widetilde{J}) = \epsilon(\widetilde{G}) + \epsilon(\widetilde{H})$.
 But the fact that $G$ is spanning implies that $J$ is, which proves (\ref{clm:computef}).
 \\

 Let $\mathcal{T}_q$ be the set of all non-trivial linked $q$-patches $H = (\widetilde{H}, L, R)$ such that $\widetilde{H} \in \mathcal{F}$, $g_q(H) = 0$ and $\widetilde{H} \cap \spn{L}$ is empty.
 Since the $q$-patches over $\mathbb{F}$ of branch-width at most $w$ are well-quasi-ordered by the minor relation, the set of minor-minimal members of $\mathcal{T}_q$ is finite; call it $\mathcal{S}_q$.
 Let $\mathcal{S} = \cup_{q = 0}^w \mathcal{S}_q$.
 Let 
 \[ P = \prod_{q = 0}^w \prod_{H \in \mathcal{S}_q} (\dim(\spn{H}) - q) \text{ and } K = \max\{ K_{\ref{thm:expand}}(H, \mathcal{F}) : H \in \mathcal{S} \}, \]
 so $K$ is the maximum of the integers $K_{\ref{thm:expand}}(H, \mathcal{F})$ given by \autoref{thm:expand} for all the patches $H$ in $\mathcal{S}$.
 Note that $P > 0$ because the patches in $\mathcal{S}$ are all non-trivial.
 We will show that $ex_\mathcal{F}(n) - \Delta n$ is periodic with period $P$ (except possibly on finitely many values of $n$).

 \begin{claim} \label{clm:nottoosmall}
  There is a positive integer $b$ such that, for every integer $q$ and every rooted configuration $H$, if $|g_q(H)| < 1/b$, then $g_q(H) = 0$.
 \end{claim}

 By \autoref{cor:rational}, $\Delta$ is a rational number; say $\Delta = a/b$ for some integers $a$ and $b$ with $b > 0$. Then $g_q(H) = \epsilon(\widetilde{H}) - \Delta(\dim(\spn{H}) - q)$ is a ratio of integers with denominator $b$, which proves (\ref{clm:nottoosmall}).
 \\

 Let $c = c_{\ref{lem:bounded}}(\mathcal{F})$ be the integer given by \autoref{lem:bounded} so that $|f(G)| < c$ for all $G$ in $\mathcal{F}$.
 Let $b$ be given by (\ref{clm:nottoosmall}).
 We set 
 \[ t = b \cdot \lceil c + \Delta w + w\rceil \text{ and } p = K P^2 |\mathcal{S}| (2t+1) + 2t .\]

 Let $N = 2^{(|\mathbb{F}|^w + 1)p^{w+1}}$.
 Recall that, by \autoref{lem:getpathdecomposition}, for any configuration $A$ in $\mathcal{F}$ with $\epsilon(A) \geq N$ there is an integer $q$ in $\{0, \ldots, w\}$ and a $q$-patch $H$ with $\widetilde{H} = A$ that is a product of $p$ non-trivial linked $q$-patches, the first of which is spanning.
 The purpose of the next three claims is to show that, for any $n > N$, $ex_\mathcal{F}(n)$ is attained by some product of the form $G_1 \times H^L \times G_2$ where $L \geq K$.

 \begin{claim} \label{clm:productofzeropatches}
  Let $m = K P^2 |\mathcal{S}|$ and let $A$ be a configuration in $\mathcal{F}$ with $\epsilon(A) \geq N$.
  There is an integer $q$ in $\{0, \ldots, w\}$ and there are $q$-patches $G, G_1, G_2, H_1, \ldots, H_m$ such that 
  \begin{itemize}
   \item $A = \widetilde{G}$,
   \item $G \in \mathcal{P}(G_1, H_1, \ldots, H_m, G_2)$,
   \item $G_1$ is spanning, and, 
   \item $H_i \in \mathcal{T}_q$ for all $i = 1, \ldots, m$.
  \end{itemize}
 \end{claim}

 Recall that, since $\epsilon(A) \geq N$, \autoref{lem:getpathdecomposition} implies that there is an integer $q$ in $\{0, \ldots, w\}$ and a sequence of $p$ non-trivial linked $q$-patches $(H_1, \ldots, H_p)$ such that $H_1$ is spanning and $A = \widetilde{G}$ for some $G$ in $\mathcal{P}(H_1, \ldots, H_p)$.
 We may assume that for each $H_i = (\widetilde{H_i}, L_i, R_i)$, if $i \geq 2$ then $\widetilde{H_i} \cap \spn{L_i}$ is empty, by moving elements $e$ of this set into $\widetilde{H_j}$ for the smallest $j$ with $e \in \spn{R_j}$.
 
 Consider any $t$ of these $q$-patches, say $H_{i_1}, \ldots, H_{i_t}$.
 The fact that all the patches $H_1, \ldots, H_p$ are linked means that $G$ has a minor $G'$ in $\mathcal{P}(H_{i_1}, \ldots, H_{i_t})$.
 We can create a spanning $q$-patch $H'_{i_1}$ out of $H_{i_1}$ by adding $q$ new elements to $\widetilde{H_{i_1}}$ parallel to the left terminals $L_{i_1}$; so $\widetilde{H_{i_1}}$ is a subconfiguration of $\widetilde{H'_{i_1}}$, $\dim(H'_{i_1}) = \dim(H_{i_1})$, and there is a patch $G''$ in $\mathcal{P}(H'_{i_1}, \ldots, H_{i_t})$ such that $\widetilde{G'}$ is a subconfiguration of $\widetilde{G''}$. So $\epsilon(\widetilde{G''}) = \epsilon(\widetilde{G'}) + q$ and $\dim(\spn{\widetilde{G''}}) - q \leq \dim(\spn{\widetilde{G'}}) \leq \dim(\spn{\widetilde{G''}})$.
 Hence
 \[ f(\widetilde{G''}) - q \leq f(\widetilde{G'}) \leq f(\widetilde{G''}) + \Delta q - q .\]
 Now, by (\ref{clm:computef}), we have 
 \begin{align*}
  f(\widetilde{G''}) &= f(\widetilde{H'_{i_1}}) + \sum_{j=2}^t g_q(H_{i_j}) \\
                     &= -\Delta q + g_q(H'_{i_1}) + \sum_{j=2}^t g_q(H_{i_j}) \label{eq:1}
 \end{align*}
 where the second equality follows from the fact that $\dim(\spn{H'_{i_1}}) = \dim(\spn{\widetilde{H'_{i_1}}})$. Then, since $g_q(H_{i_1}) \leq g_q(H'_{i_1}) \leq g_q(H_{i_1}) + q$, we have 
 \[ -\Delta q + \sum_{j=1}^t g_q(H_{i_j}) \leq f(\widetilde{G''}) \leq -\Delta q + q + \sum_{j=1}^t g_q(H_{i_j}), \]
 and so
 \[ -q - \Delta q + \sum_{j=1}^t g_q(H_{i_j}) \leq f(\widetilde{G'}) \leq \sum_{j=1}^t g_q(H_{i_j}) .\]
 If $g_q(H_{i_j}) > 0$ for all $j = 1, \ldots, t$, then by (\ref{clm:nottoosmall}) and the definition of $t$, it follows that $f(\widetilde{G'}) \geq c$, a contradiction to \autoref{lem:bounded}.
 On the other hand, if $g_q(H_{i_j}) < 0$ for all $j = 1, \ldots, t$, then by (\ref{clm:nottoosmall}) and the definition of $t$, it follows that $f(\widetilde{G'}) \leq -c$, a contradiction to \autoref{lem:bounded}.
 
 This proves that $g_q(H_i) > 0$ for fewer than $t$ of the patches $H_i$ and that $g_q(H_i) < 0$ for fewer than $t$ of the patches $H_i$.
 Hence all but at most $2t$ of the patches $H_1, \ldots, H_p$ are in $\mathcal{T}_q$.

 Note that $m = (p-2t)/(2t+1)$. So there is a stretch $H_{k+1}, H_{k+2}, \ldots, H_{k+m}$ of these patches such that $H_{k+1}, \ldots, H_{k+m} \in \mathcal{T}_q$.
 Since $G \in \mathcal{P}(H_1, \ldots, H_p)$, there are elements $G_1$ of $\mathcal{P}(H_1, \ldots, H_k)$ and $G_2$ of $\mathcal{P}(H_{k+m+1}, \ldots, H_p)$ such that $G \in \mathcal{P}(G_1, H_{k+1}, \ldots, H_{k+m}, G_2)$. The fact that $H_1$ is spanning implies that $G_1$ is.
 This proves (\ref{clm:productofzeropatches}).

 \begin{claim} \label{clm:productofminimalpatches}
  Let $s = K P |\mathcal{S}|$ and let $A$ be a configuration in $\mathcal{F}$ with $\epsilon(A) \geq N$.
  There is an integer $q$ in $\{0, \ldots, w\}$ and there are $q$-patches $G, G_1, G_2, H_1, \ldots, H_s$ such that
  \begin{itemize}
   \item $\widetilde{G}$ is a minor of $A$,
   \item $G \in \mathcal{P}(G_1, H_1, \ldots, H_s, G_2)$,
   \item $G_1$ is spanning,
   \item $H_1, \ldots, H_s \in \mathcal{S}_q$, and   
   \item $f(\widetilde{G}) = f(A)$ and $\dim(\spn{G}) \equiv \dim(\spn{A}) \pmod P$.
  \end{itemize}
 \end{claim}

 Consider the integer $q$ and the $q$-patches $G, G_1, G_2, H_1, \ldots, H_m$ given by (\ref{clm:productofzeropatches}).
 For $i = 1, \ldots, m$, the $q$-patch $H_i$ has a minor $H'_i$ in $\mathcal{S}_q$.
 We set $a_i = \dim(\spn{H_i}) - \dim(\spn{H'_i})$.
 Since $m = sP$, it follows from \autoref{lem:getzerosubsequence} that there is some subsequence $a_{j+1}, \ldots, a_{j+s'}$ such that $s' \geq s$ and $\sum_{i=j+1}^{j+s'} a_i \equiv 0 \pmod P$. 

 Then $G$ has a minor $G'$ such that
 \[ G' \in \mathcal{P}(G_1, H_1, \ldots, H_j, H'_{j+1}, \ldots, H'_{j+s'}, H_{j+s'+1}, \ldots, H_m, G_2) \]
 and such that $\dim(\spn{G'}) \equiv \dim(\spn{G}) \pmod P$. Since $G_1$ is spanning, so are $G$ and $G'$, so $\dim(\spn{\widetilde{G'}}) \equiv \dim(\spn{A})  \pmod P$. Also, $\widetilde{G'}$ is a minor of $A$.

 We also have $f(\widetilde{G'}) = f(\widetilde{G}) = f(A)$ by (\ref{clm:computef}) because $g_q(H_i) = g_q(H'_i) = 0$ for all $i$.
 There are $q$-patches $G'_1$ in $\mathcal{P}(G_1, H_1, \ldots, H_j)$ and $G'_2$ in $\mathcal{P}(H'_{j+s+1}, \ldots, H_m, G_2)$ such that $G' \in \mathcal{P}(G'_1, H'_{j+1}, \ldots, H'_{j+s}, G'_2)$.
 Note that $G'_1$ is spanning because $G_1$ is.
 So (\ref{clm:productofminimalpatches}) holds with $G', G'_1, G'_2, H'_{j+1}, \ldots, H'_{j+s}$ in place of $G, G_1, G_2, H_1, \ldots, H_s$. 
 
 \begin{claim} \label{clm:repeatedpatch}
  Let $A$ be a configuration in $\mathcal{F}$ with $\epsilon(A) \geq N$.
  There is an integer $q$ in $\{0, \ldots, w\}$, $q$-patches $G, G_1, G_2, H$, and an integer $K'$ such that $K' \geq K$,
  \begin{itemize}
   \item $\widetilde{G}$ is a minor of $A$,
   \item $G \in \mathcal{P}(G_1, H^{K'}, G_2)$,
   \item $G_1$ is spanning,
   \item $H \in \mathcal{S}_q$, and
   \item $f(\widetilde{G}) = f(A)$ and $\dim(\spn{G}) \equiv \dim(\spn{A}) \pmod P$.
  \end{itemize}
 \end{claim}

 Consider the integer $q$ and the $q$-patches $G, G_1, G_2, H_1, \ldots, H_s$ given by (\ref{clm:productofminimalpatches}).
 Since $\mathcal{S}_q$ is finite, the patches $H_1, \ldots, H_s$ fall into at most $|\mathcal{S}_q|$ isomorphism classes.
 There is a $q$-patch $H$ in $\mathcal{S}_q$ so that at least $s / |\mathcal{S}_q| \geq s / |\mathcal{S}| = KP$ of these patches are isomorphic to $H$; let $H_{i_1}, \ldots, H_{i_{KP}}$ be a subsequence consisting of some $KP$ of them.

 Define the sequence $a_1, \ldots, a_{KP-1}$ by setting 
 \[ a_j = \sum_{\ell = i_j + 1}^{i_{j+1} - 1} (\dim(\spn{H_\ell}) - q) \]
 for each $j = 1, \ldots, KP-1$.
 By \autoref{lem:getzerosubsequence}, it has a subsequence $a_{m+1}, \ldots, a_{m+L}$ where $\sum_{\ell=m+1}^{m+L} a_\ell \equiv 0 \pmod P$ for some $L \geq K-1$.
 
 Since all patches in $\mathcal{S}_q$ are linked, $G$ has a minor $G' = G'_1 \times G^* \times G'_2$ where
 \begin{align*}
  G'_1 &\in \mathcal{P}(G_1, H_1, \ldots, H_{i_{m+1} - 1}), \\
  G^* &\in \mathcal{P}(H_{i_{m+1}}, H_{i_{m+2}}, H_{i_{m+3}}, \ldots, H_{i_{m+L+1}}), \\
  G'_2 &\in \mathcal{P}(H_{i_{m+L+1}+1}, \ldots, H_s, G_2).
 \end{align*}
 That is, we have taken the product defining $G$ and removed from it all the terms $H_j$ where $j$ lies in the interval $(i_{m+1}, i_{m+L+1})$ and is not actually one of the values $i_{m+1}, i_{m+2}, \ldots, i_{m+L}, i_{m+L+1}$.
 Note that $\widetilde{G'}$ is a minor of $A$.
 Since each of $H_{i_{m+1}}, \ldots, H_{i_{m+L+1}}$ is isomorphic to $H$, we have
 $G^* \in \mathcal{P}(H^{L+1})$ and $G' \in \mathcal{P}(G_1, H_1, \ldots, H_{i_{m+1} - 1}, H^{L+1}, H_{i_{m+L+1}+1}, \ldots, H_s, G_2)$.
  
 We have $\dim(\spn{G'}) \equiv \dim(\spn{G}) \equiv \dim(\spn{A}) \pmod P$ because $\sum_{\ell=m+1}^{m+L} a_\ell \equiv 0 \pmod P$.
 Since $G_1$ is spanning, so are $G$, $G'_1$, and $G'$. Therefore, it follows from the fact that $g_q(H_i) = 0$ for all patches $H_i$ that $f(\widetilde{G'}) = f(\widetilde{G}) = f(A)$.
 Since $G' \in \mathcal{P}(G'_1, H^{L+1}, G'_2)$, claim (\ref{clm:repeatedpatch}) holds with $G', G'_1, G'_2, L+1$ in place of $G, G_1, G_2, K'$.
 \\

 We are now equipped to finish the proof.
 Define the function $f'$ by setting $f'(n) = ex_\mathcal{F}(n) - \Delta n$ for all $n$.
 Since $f'$ is bounded (\autoref{lem:bounded}), there is an integer $M$ such that, for each $i$ in $\{0, \ldots, P-1\}$, 
 \[ \max\{f'(n) : n > M, n \equiv i \pmod P \} = f'(n_i) \]
 for some $n_i$ with $n_i \equiv i \pmod P$ and $N < n_i < M$.
 
 Fix an integer $i$ in $\{0, \ldots, P-1\}$.
 Let $A$ be a configuration in $\mathcal{F}$ maximizing $f(A)$ subject to $\dim(\spn{A}) = n_i$. So $f(A) = f'(n_i)$.
 
 We have $\epsilon(A) \geq \dim(\spn{A}) = n_i > N$.
 So we can apply (\ref{clm:repeatedpatch}); let $q, G, G_1, G_2, H$ and $K'$ be as given.
 Then $\dim(\spn{G}) \equiv n_i \equiv i \pmod P$.

 Let $\mathcal{F}^*$ be the set of $q$-patches $G$ such that $\widetilde{G} \in \mathcal{F}$. It is minor closed.
 Thus by the definition of $K$ and \autoref{thm:expand} applied to $\mathcal{F}^*$, it follows that for any $L \geq 0$ there is an element $G'_L$ of $\mathcal{P}(G_1, H^L, G_2)$ with $\widetilde{G'_L} \in \mathcal{F}$.
 
 The fact that $H \in \mathcal{S}_q$ means that $\dim(\spn{H}) - q$ divides $P$. This means that for any integer $n$ such that $n > M > n_i$ and $n \equiv \dim(\spn{G}) \equiv i \pmod P$, there is an integer $L$ and an element $F_n = G'_L$ of $\mathcal{P}(G_1, H^L, G_2)$ with $\dim(\spn{F_n}) = n$ and $\widetilde{F_n} \in \mathcal{F}$.
 Since $g_q(H) = 0$ and $G_1$ is spanning, (\ref{clm:computef}) implies that $f(\widetilde{F_n}) = f(\widetilde{G}) = f(A) = f'(n_i)$. Hence $f(\widetilde{F_n}) \geq f'(n) = ex_\mathcal{F}(n) - \Delta n$.
 Since $F_n$ is spanning, $\dim(\spn{\widetilde{F_n}}) = n$, and it follows that $\epsilon(\widetilde{F_n}) = ex_\mathcal{F}(n)$.
\end{proof}

We can easily prove \autoref{thm:maintheorem} as a corollary of \autoref{thm:path}.

\begin{proof}[Proof of \autoref{thm:maintheorem}]
 It is equivalent to prove the theorem for a minor-closed class of configurations $\mathcal{F}$ over $\mathbb{F}$ of bounded branch-width: let $\mathcal{F}$ be the closure under minors of the set of configurations $\{A : M(A) \in \mathcal{M}\}$ (we only need to explicitly close this under minors because if $A'$ is a minor of $A$, then $M(A')$ may contain loops that are not present in the corresponding minor of $M(A)$).
 
 Applying \autoref{thm:path}, there are integers $p$ and $m$ such that, for each $i$ in $\{0, \ldots, p-1\}$ there is an integer $q$ and $q$-patches $G_1, H, G_2$  such that whenever $n$ is an integer congruent to $i$ mod $p$ and $n > m$, there is an integer $L$ and a spanning $q$-patch $F$ in $\mathcal{P}(G_1, H^L, G_2)$ such that $\widetilde{F} \in \mathcal{F}$, $\dim(\spn{F}) = n$, and $\epsilon(\widetilde{F}) = ex_\mathcal{F}(n)$.
 
 Fix an $i$ in $\{0, \ldots, p-1\}$ and consider the resulting integer $q$ and $q$-patches $G_1, H, G_2$. Let $n$ be an integer congruent to $i$ mod $p$ with $n > m$.
 Let $\Delta = \epsilon(\widetilde{H}) / (\dim(\spn{H}) - q)$. 
 We have an integer $L$ and a spanning $q$-patch $F$ in $\mathcal{P}(G_1, H^L, G_2)$ such that $\widetilde{F} \in \mathcal{F}$, $\dim(\spn{\widetilde{F}}) = n$, and $\epsilon(\widetilde{F}) = ex_\mathcal{F}(n)$. Then
 \begin{align*}
  \epsilon(\widetilde{F}) &= \epsilon(\widetilde{G_1}) + \epsilon(\widetilde{G_2}) + L \epsilon(\widetilde{H}) \\
                          &= \epsilon(\widetilde{G_1}) + \epsilon(\widetilde{G_2}) + L \Delta(\dim(\spn{H}) - q). \\
 \intertext{ Also, }
  n &= \dim(\spn{\widetilde{F}}) = \dim(\spn{F}) \\
    &= \dim(\spn{G_1}) + L(\dim(\spn{H})- q) + \dim(\spn{G_2}) - q.
 \end{align*}
 Therefore, $\epsilon(\widetilde{F}) = \epsilon(\widetilde{G_1}) + \epsilon(\widetilde{G_2}) + \Delta (n - \dim(\spn{G_1}) - \dim(\spn{G_2}) + q)$.
 So the theorem follows by setting $a_i = \epsilon(\widetilde{G_1}) + \epsilon(\widetilde{G_2}) - \Delta(\dim(\spn{G_1}) + \dim(\spn{G_2}) - q)$.
\end{proof}

\section*{Acknowledgements}

I thank Sergey Norin, with whom I have been working on a wider project on densities of minor-closed classes of graphs, for sharing many ideas on this topic that could also be applied to matroids.
I also thank Jim Geelen for telling me how to prove \autoref*{thm:rootedwqo}.


\begin{thebibliography}{9}

 \bibitem{Eppstein} David Eppstein, ``Densities of minor-closed graph families'', \emph{Electron. J. Combin.} 17(1), Paper R136, 2010.

 \bibitem{GeelenGerardsWhittle} James F. Geelen, A. M. H. Gerards, Geoff Whittle, ``Branch-width and well-quasi-ordering in matroids and graphs'', \emph{J. Combin. Theory Ser. B} \textbf{84} (2002), 270--290.

 \bibitem{GeelenGerardsWhittle2007} Jim Geelen, Bert Gerards, and Geoff Whittle, ``Excluding a planar graph from GF$(q)$-representable matroids'', \emph{J. Combin. Theory Ser. B} \textbf{97} (2007), 971--998.
 
 \bibitem{GeelenGerardsWhittle2006} Jim Geelen, Bert Gerards, and Geoff Whittle, ``On Rota's conjecture and excluded minors containing large projective geometries'', \emph{J. Combin. Theory Ser. B} \textbf{96} (2006), 405--425.
 
 \bibitem{GeelenGerardsWhittle2015}  Jim Geelen, Bert Gerards, Geoff Whittle, ``The Highly Connected Matroids in Minor-Closed Classes'', \emph{Ann. Comb.} \textbf{19} (2015), 107--123.
 
 \bibitem{GeelenWhittle} James Geelen and Geoff Whittle, ``Cliques in dense GF$(q)$-representable matroids'', \emph{J. Combin. Theory Ser. B} \textbf{87} (2003), 264--269.

 \bibitem{KapadiaNorin} Rohan Kapadia and Sergey Norin, \emph{in preparation}.
 
 \bibitem{Mader} W. Mader, ``Homomorphieeigenschaften und mittlere Kantendichte von Graphen'', \emph{Math. Ann.} \textbf{174} (1967), 265--268.

 \bibitem{SongThomas} Zi-Xia Song and Robin Thomas, ``The extremal function for $K_9$ minors'', \emph{J. Combin. Theory Ser. B} \textbf{96} (2006), 240--252.
 
\end{thebibliography}
\end{document}